\documentclass[12pt]{article}
\usepackage{amsfonts,amsmath,color,amsthm,amssymb}
\usepackage{graphicx}
\usepackage{tikz}
\usepackage{pstricks}
\usepackage{graphicx}
\usepackage{epstopdf}
\usepackage{permute}
\usepackage{latexsym}
\usepackage{mathrsfs}
\usepackage{hyperref}
\usepackage[utf8]{inputenc}
\usepackage{lscape}
\usepackage{longtable}
\usepackage{algorithm2e} 
\usepackage{enumitem} 
\usepackage{array}
\usepackage{tikz}
\usepackage{tikz-qtree,tikz-qtree-compat}
\usetikzlibrary{positioning,decorations.pathreplacing}
\title{Graph Universal Cycles of Combinatorial Objects}

\author{Amelia Cantwell\\
Department of Mathematics\\ University of Montana\and 
Juliann Geraci\\
Department of Mathematics\\ State University of New York, Oswego\and Anant Godbole\\
Department of Mathematics and Statistics\\ East Tennessee State University\and
Cristobal Padilla\\
Department of Mathematics\\San Francisco State University
}

\newtheorem{thm}{Theorem}[section]

\newtheorem{prop}[thm]{Proposition}

\newtheorem{dfn}{Definition}[section]

\def\cp{\mathcal P}
\def\cf{\mathcal F}
\def\cw{\mathcal W}
\def\n{\noindent}
\date{}
\begin{document}
\maketitle
\begin{abstract}A connected digraph in which the in-degree of any vertex equals its out-degree is Eulerian; this baseline result is used as the basis of existence proofs for universal cycles (also known as ucycles or generalized deBruijn cycles or U-cycles) of several combinatorial objects.   The existence of ucycles is often dependent on the specific representation that we use for the combinatorial objects.  For example, should we represent the subset $\{2,5\}$ of $\{1,2,3,4,5\}$ as ``25" in a linear string?  Is the representation ``52" acceptable?  Or it it tactically advantageous (and acceptable) to go with $\{0,1,0,0,1\}$?  In this paper, we represent combinatorial objects as graphs, as in \cite{bks}, and exhibit the flexibility and power of this representation to produce {\it graph universal cycles}, or {\it Gucycles}, for $k$-subsets of an $n$-set; permutations (and classes of permutations) of $[n]=\{1,2,\ldots,n\}$, and partitions of an $n$-set, thus revisiting the classes first studied in \cite{cdg}.    Under this graphical scheme, we will represent $\{2,5\}$ as the subgraph $A$ of $C_5$ with edge set consisting of $\{2,3\}$ and $\{5,1\}$, namely the ``second" and ``fifth" edges in $C_5$.  Permutations are represented via their permutation graphs, and set partitions through disjoint unions of complete graphs.
\end{abstract}
\section{Introduction}  

A somewhat loose definition of universal cycles was used in a recent VCU seminar talk by Glenn Hurlbert in which he stated that ``Broadly, universal cycles are special listings of combinatorial objects
in which {\it codes for the objects} are written in an overlapping, cyclic manner."  By ``special" Hurlbert means ``without repetitions", i.e., so that each linear window of a specific length represents a different object.  We stress that the ``codes for the objects" are not pre-dictated but are often chosen in a conventional way.   See \cite{b} and \cite{dgg} for exhaustive treatments.  We next give some examples:

\medskip

\n EXAMPLE 1: The cyclic string 112233 encodes each of the six multisets of size 2 from the set $\{1,2,3\}$, with, e.g., 31 and 22 representing $\{1,3\}$ and $\{2,2\}$ respectively.  The window size is 2.

\medskip

\n EXAMPLE 2: The string 11101000 encodes each of the binary three-letter words in the obvious way, but is also a ucycle of the eight subsets of $\{1,2,3\}$, with the binary string coding -- in which membership in the set is indicated by a 1, e.g., 101 represents the subset $\{1,3\}$.  The string can also be a representation of all subgraphs of the complete graph $K_3$.  The window length is 3.

\medskip

\n EXAMPLE 3: The binary string  1110011010 is a ucycle of all subsets of size 2 and 3 of a 4-element set, using a window of length 4, and the binary string coding.

\medskip

\n EXAMPLE 4 (\cite{h}):  \[ 1356725\ 6823472\ 3578147\ 8245614\ 5712361\ 2467836\ 7134582\ 4681258, \]
where each block is obtained from the previous one by addition of 5 modulo 8, is an encoding of the 56 3-sets of the set $[8]$.  The window is of length 3, and, e.g.,  the block 836 represents the subset $\{3,6,8\}$.

\medskip

\n EXAMPLE 5:  The string 124324 encodes each of the six permutations of $\{1,2,3\}$ in an order isomorphic fashion (e.g. 243 represents the permutation 132). It is clearly not possible to create a ucycle of length six using the ground set $\{1,2,3\}$.

\medskip

\n EXAMPLE 6 (\cite{cdg}): We have the ucycle $abcbccccddcdeec$ of $\cp(4)$, the set of all partitions of $\{1,2,3,4\}$ into an arbitrary number of parts, where, for example, the string $dcde$ encodes the partition $13\vert2\vert4$.  Note that the alphabet used was in this case of size 5, though an alphabet of (minimum) size 5 is shown to suffice to encode $\cp(5)$ as 
$$DDDDDCHHHCCDDCCCHCHCSHHSDSSDSSHSDDCH$$$$SSCHSHDHSCHSJCDC.$$

\medskip

A ucycle is usually shown to exist by showing that an arc digraph $D$ is Eulerian, which in turn holds if it is (a) balanced (i.e., the indegree $i(v)$ of every vertex $v\in D$ equals its outdegree $o(v)$) and (b) weakly connected.  Weak connectedness is often showed by exhibiting a path from any starting vertex to a strategically chosen sink vertex.  The edge set of the arc digraph consists of the objects that we are trying to ucycle, and the vertices are most often taken to be the ``overlaps" between consecutive edges.  Alternately, edges are labeled as the concatenation of adjacent vertex labels.  We illustrate this stategy by showing that the set of $n$-letter words on a $k$-letter alphabet admits a ucycle (this is the classical deBruijn theorem).  Vertices of $D$ are $(n-1)$-letter words with an edge from $v_1$ to $v_2$ if the last $(n-2)$ letters of $v_1$ coincide with the first $n-2$ letters of $v_2$.  The edge label, obtained by concatenation, are the desired objects we seek to ucycle.  Connectedness is  easy to establish, and in- and out-degrees may both be seen to be $k$, so $D$ is Eulerian and the Eulerian cycle spells out the ucycle.  The arc digraph that leads to the ucycle in Example 2 is given in Figure 1.

\begin{figure}[h]
\centering 
\includegraphics[width=0.3\textwidth]{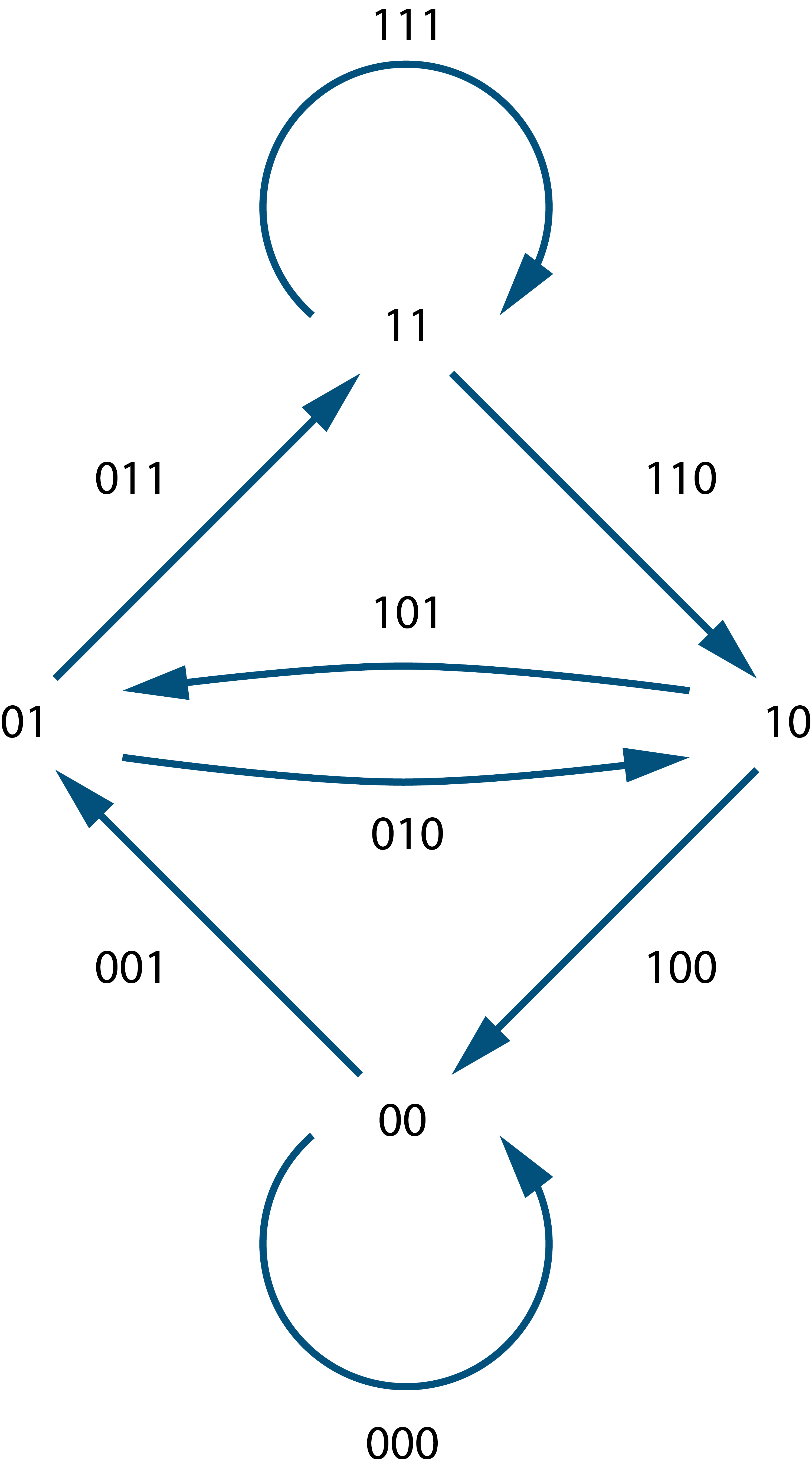}
\caption{The arc digraph for Example 2}
\end{figure}

In the deBruijn theorem example above, we can alternately think of the vertices as multisets of $[n-1]$, with each entry appearing between $0$ and $(k-1)$-times, and with, e.g.,  a word such as 30221 ($n=6, k=4$) representing the multiset $\{1,1,1,3,3,4,4,5\}$.  Moreover, under this interpretation, there is an edge from $v_1$ to $v_2$ if the multiset frequencies of $\{2,\ldots,n-1\}$ in $v_1$ coincide with the multiset frequencies of $\{1,\ldots,n-2\}$ in $v_2$, and the concatenated edge is a multiset of $[n]$, with each element occurring $\le k-1$ times.  This notion of a window shift necessitating a relabeling of the vertices is crucial in the context of graph universal cycles, Gucycles, which we turn to next.
We start with some key definitions from \cite{bks} and an example:

\begin{dfn}
Given a labeled graph $G$ with vertex set $V(G) = \{v_1, v_2, . . . , v_N\}$ with
vertices labeled by the rule $v_j\rightarrow j$ and an integer $n\in\{1,\ldots,N-1\}$, an {\bf $n$-window of $G$} is
the subgraph of $G$ {\bf induced} by the vertex set $V = \{v_i, v_{i+1},\ldots,v_{i+n-1}\}$ for some $i$, where
vertex subscripts are reduced modulo $N$ as appropriate, and vertices are {\bf relabeled} such that
$v_i\rightarrow 1; v_{i+1}\rightarrow 2,\ldots v_{i+n-1}\rightarrow n$.  For each $i\in \{1,\ldots,N\}$, we denote
the corresponding $i$th $n$-window of $G$ as $W_{G,n}(i)$, or as $W_n(i)$ if $G$ is clear from the context.
\end{dfn}

\begin{dfn} Given $\cf$, a family of labeled graphs on $n$ vertices, a {\bf graph universal cycle}
(Gucycle) of $\cf$, is a labeled graph $G$ on $N$ vertices such that the sequence of $n$-windows of $G$ contains each
graph in $\cf$ precisely once. That is, $\{W_n(i)|1\le i\le N\}=\cf$, and $W_n(i) = W_n(j)\Rightarrow i = j$.
\end{dfn}

Brockman et al. \cite{bks} prove the existence of Gucycles of classes of labeled graphs on $n$ vertices, including all simple graphs, trees, graphs with $k$ edges, graphs with loops, graphs with multiple edges (with up to $d$ duplications of each edge), directed graphs, hypergraphs, and $r$-uniform hypergraphs.  Figure 2 shows the Gucycle of all graphs on 3 labeled vertices which, in turn, is constructed from the arc digraph in Figure 3:

\medskip

\begin{figure}[h]
\centering 
\includegraphics[width=0.7\textwidth]{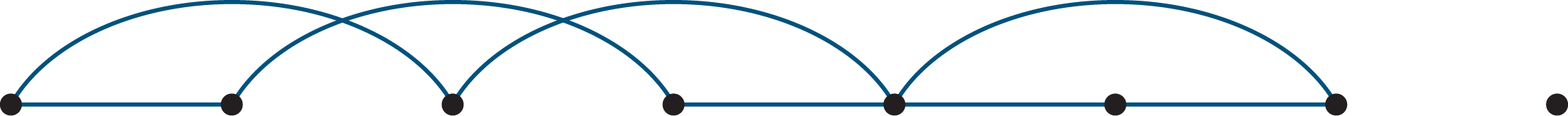}
\caption{Gucycle of all labeled graphs on 3 vertices.}
\end{figure}

\medskip

\begin{figure}[h]
\centering
\includegraphics[width=0.5\textwidth]{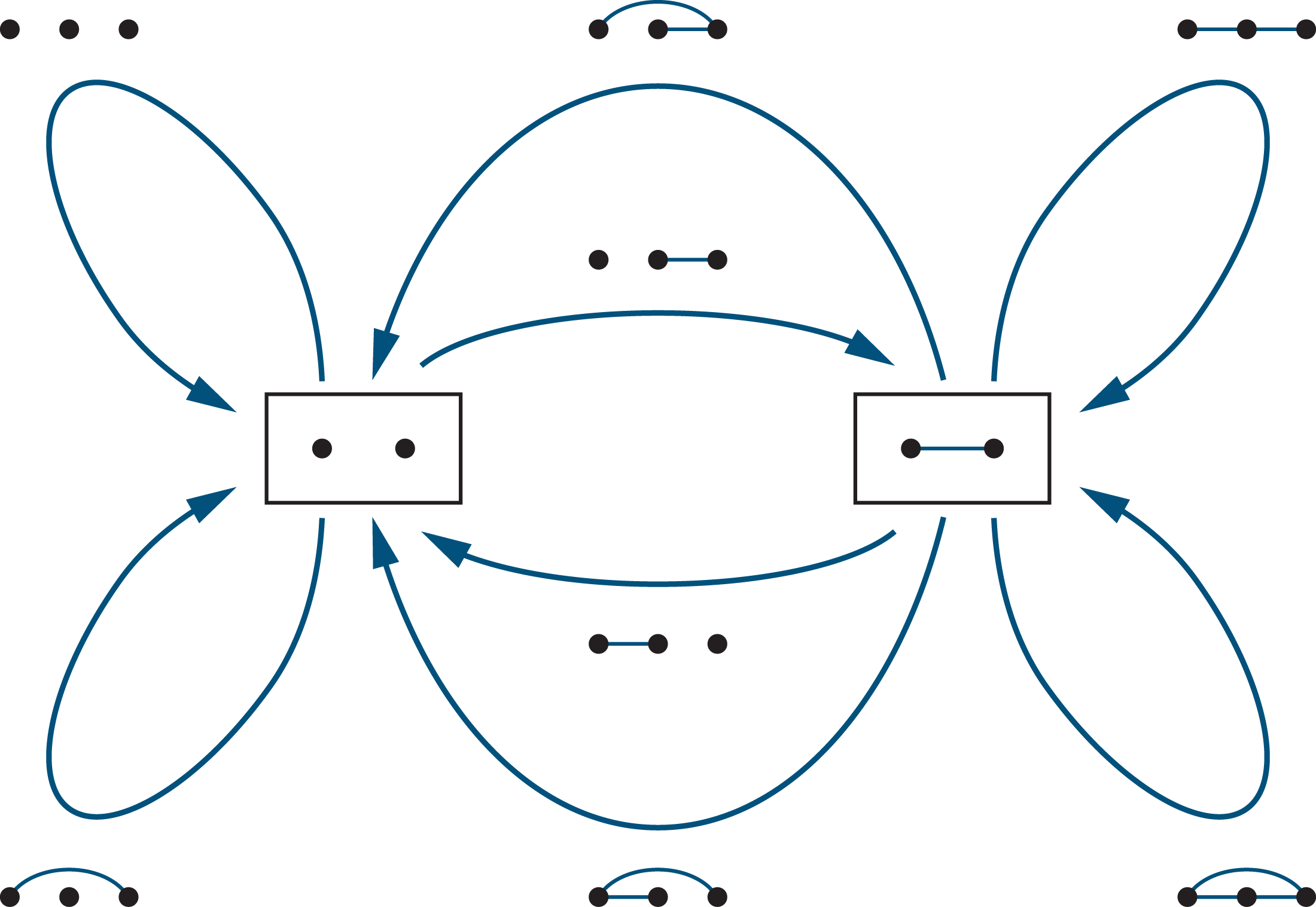}
\caption{The arc digraph that produces the Gucycle in Figure 2.}
\end{figure}

%
%
%
%
In this paper, we extend the family of classes $\cf$ that admit Gucycles to
\begin{itemize}
\item Subgraphs of size $k$ of the cycle $C_n$.  This leads to Theorem 2.5, namely the existence of Gucycles for $k$-subsets of an $n$-set, {\em for all values of} $k,n$.
\item A class of multiple edge subgraphs of $C_n$, which yields Theorem 2.6 (Gucycles for $k$-multisets of an $n$-set $\forall k,n$).
\item Permutation graphs on $n$ vertices.  This leads to Theorem 3.1, which exhibits the existence of Gucycles for all permutations of $[n]$.
\item Transposition graphs on $n$ vertices, leading to the existence of Gucycles for permutation involutions (Theorem 3.2).
\item Disjoint unions of cliques on $n$ vertices.  This leads to Theorem 4.1 on Gucycles for set partitions.
\end{itemize}
\section{Subsets and Multisets}

If we are to represent $k$-subsets of $[n]$ via their elements, as in Example 4 of Section 1, then we clearly must have each element of $[n]$ appear the same number of times in the ucycle, and thus
\[n\big\vert{n\choose k}\]
is a necessary condition for the existence of a ucycle.  In \cite{cdg}, Chung et al. conjectured that for each $k$, there exists an $n_0(k)$ such that ucycles exist for $k$-subsets of $[n]$ provided that $n\ge n_0(k)$ is such that the divisibility condition above holds.  This conjecture was proved by \cite {gjko}.  For work prior to \cite{gjko} see also \cite{j} and \cite{r} in addition to \cite{cdg} and \cite{h}.  

We are more interested, however, in binary string-type codings that dispense with divisibility conditions, and next guide the discussion in this direction.

Ucycles for $k$-subsets of $[n]$ cannot exist in the binary string coding unless $k=n-1$ when we have the ucycle $111\ldots10$; for other values of $k$,  we are forced into an incomplete cycle as seen by the example 
\[a=110000\rightarrow100001\rightarrow000011\rightarrow000110\rightarrow001100\rightarrow011000\rightarrow a\]
for $k=2; n=6$.

Several authors have produced results that almost complete the quest of producing ucycles for ${{\bf n}\choose {\bf k}}$, the set of all $k$-subsets of $n$.  For example, in \cite{hh}, the authors prove that under certain conditions, there exists an $s$-ocycle of all the permutations of a fixed multiset, where an ocycle, or overlap cycle  is one in which the overlap between consecutive $k$ windows is not $k-1$ -- but rather equals $s$.  If we take the multiset to be one with $k$ ones and $n-k$ zeros we see that the following is true
\begin{prop} (\cite{hh}) There exist, for $1\le s\le n-2$; $\gcd(n,s)=1$, $s$-ocycles of ${{\bf n}\choose {\bf k}}$ in the binary string coding.
\end{prop}
In another paper \cite{chhm}, we see that
\begin{prop} (\cite{chhm}) Universal packings of length ${n\choose k}(1-o(1))$ exist for each $k$.
\end{prop}
Note that there is no divisibility condition that hampers the universality of Propositions 2.1 and 2.2.  The same is true of the next result from \cite{bg}:
\begin{prop} (\cite{bg}) For each $1\le s<t\le n$ there exists a ucycle of binary words of length $n$ with between $s$ and $t$ ones, i.e., of subsets of size in the range $[s,t]$, in the binary string coding.
\end{prop}
It is the above result that will prove to be critical in the proof of Theorem 2.5.

Now we have seen above that ucycles of 2-subsets of $[6]$ do not exist in the binary string coding nor in the traditional sense, since 6 does not divide ${6 \choose 2}$.  Consider, however, the $15$-vertex graph in Figure 3, which is a Gucycle of all 2-edge subsets of $K_4$.  It is not suprising that this Gucycle exists -- by virtue of Theorem 3.5 in \cite{bks}, which states that Gucycles exist for all graphs on $n$ vertices with precisely $k$ edges.  If we next label the six edges of $K_4$  lexicographically according to the code
\[\{1,2\}=1; \{1,3\}=2; \{1,4\}=3; \{2,3\}=4, \{2,4\}=5; \{3,4\}=6,\]  we see that the Gucycle in Figure 4 can be viewed as a Gucycle of all 2-subsets of $[6]$! 

\begin{figure}[h]
\centering
\includegraphics[width=0.7\textwidth]{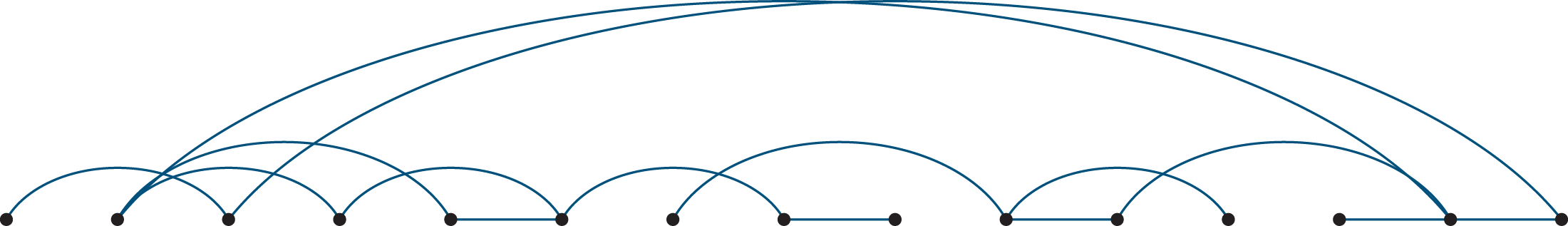}
\caption{Gucycle of the 2-subsets of $[6]$.}
\end{figure}

Since $K_n$ has ${n\choose2}$ edges, we quickly see that the following is a corollary of the above theorem, and a special case of Theorem 2.5 below:
\begin{prop}
There exists a Gucycle of all $k$-subsets of an $N$-element set, where $N={n\choose2}$ for some $n$.
\end{prop}

%
%
%
%

Notice that in the above Gucycle each of the six edges of $K_4$ appears five times, even though a specific edge, e.g., the one joining $v_5$ and $v_6$ plays the role of a $\{1,2\}$ edge, a $\{2,3\}$ edge, and a $\{3,4\}$ edge, and thus represents the numbers 1, 4, and 6.  On the other hand the edge joining $v_2$ and $v_5$ only plays the role of the number 3 (i.e., the edge $\{1,4\}$).  By contrast, in the tradititional approach to ucycles of $k$-element subsets of $[n]$, an element such as ``4" always represents itself, i.e., the number 4.

\medskip

\begin{thm}For each $1\le k\le n$ there exists a Gucycle for $k$-element subsets of $[n]$.
\end{thm}
\begin{proof}  By Proposition 2.3, there exists a binary string coding ucycle of all $k-1$ and $k$-subsets of $[n-1]$.  Label the edges of $C_n$ as $e_1=\{1,2\}, e_2=\{2,3\}, \ldots, e_n=\{n,1\}$.  Identify the binary string of the $k$-subset $A$ of $[n-1]$, consisting of ones in the $r_i$th positions; $1\le i\le k$ with the edges $\{e_i:1\le i\le k\}$.  Identify the binary string of the $(k-1)$-subset $A$ of $[n-1]$, consisting of ones in the $r_i$th positions; $1\le i\le k-1$ with the edges $\{e_i:1\le i\le k-1; e_n\}$ (where here, in the punch line of the proof, we are identifying a $k-1$ subset of $[n-1]$ with a $k$-subgraph of $C_n$.)  Since the binary string coding relies on a simple right shift, and since the edge $e_n$, if present, is not part of the graph induced by the right shift, we see that there is a bijection between binary-string coded subsets of size $k-1$ or $k$, of $[n-1]$, and the $k$-subgraphs of $C_n$.  The latter, in turn, can be identified with the collection of $k$-subsets of $[n]$, via the bijection $j\rightarrow e_j$.
\end{proof}
Figure 5 shows how the ucycle of Example 3, Section 1, translates to a Gucycle of the 3-subsets of $[5]$.  

\begin{figure}[h]
\centering
\includegraphics[width=0.7\textwidth]{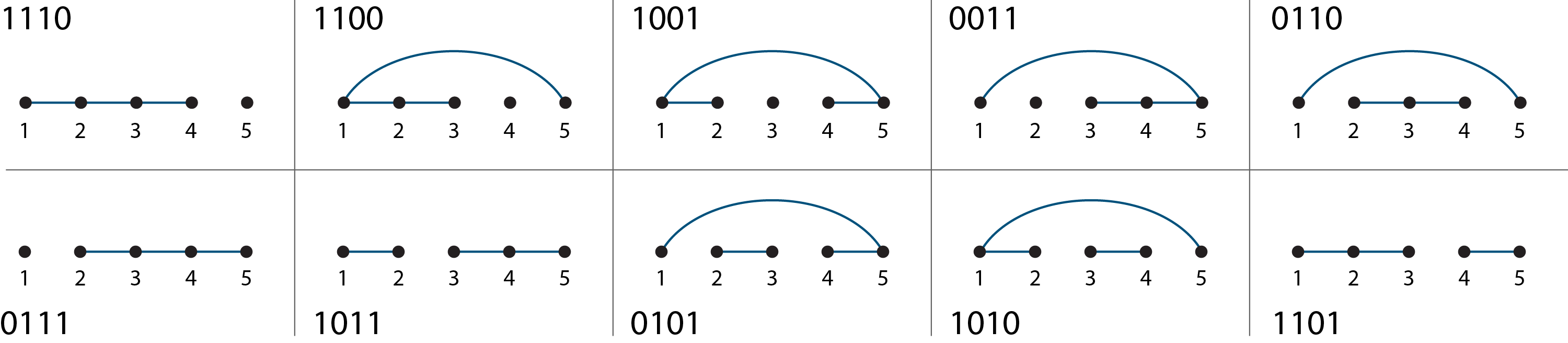}
\caption{Gucycle windows of the 3-subsets of [5]}
\end{figure}

%
%
%

\begin{thm}For each $k$ and $n\ge 2$, there exists a Gucycle for $k$-multisets of $[n]$.
\end{thm}
\begin{proof}  We start with the following result in \cite{cgk}, which improves on  Theorem 4.1 in \cite{bg}:
\begin{prop}  Let $k,n\in{\mathbb Z}^+$.  Consider $n$ letter words $w=(w_1,w_2,\ldots,w_n)$ on the $k+1$-letter alphabet $\Lambda=\{0,1,\ldots,k\}$, and define the weight $h(w)$ of $w$ by $h(w)=\sum_{i=1}^nw_i$.  Let $s,t\in{\mathbb Z}^+$ satisfy $0\le s<s+k\le t\le nk$. Let $\cw$ be the collection of all such words with weights between $s$ and $t$. Then there exists a ucycle of the elements of $\cw$.
\end{prop}
\n To prove Theorem 2.6, we will let $s=0$ and $t=k$ in Proposition 2.7.  By Proposition 2.7, for $n\ge 2$ there exists an alphabet coding ucycle of all multisets of $[n-1]$ of size between 0 and $k$.  Label the potential multi-edges of $C_n$ as $e_1=\{1,2\}, e_2=\{2,3\}, \ldots, e_n=\{n,1\}$.  Identify the alphabet string of the $k$-multiset $A$ of $[n-1]$, consisting of the number $a_i$ in the $r_i$th positions; $1\le i\le n-1$ with the multi-edge sets $\{e_i:1\le i\le n-1\}$, where $e_i$ contains $a_i$ edges.  Identify the alphabet string of the $m$-multiset $A$ of $[n-1]$ $(0\le m\le k-1)$, consisting of $a_i$ in the $r_i$th positions; with the multiedges edges $\{e_i:1\le i\le n-1; e_n\}$ (where $e_i$ contains $a_i$ edges for $1\le i\le n-1$ and $e_n$ contains $k-m$ edges).  Since the alphabet string coding relies on a simple right shift, and since the edge $e_n$, if present, is not part of the graph induced by the right shift, we see that there is a bijection between alphabet-string coded multisets of size between 0 and $k$, of $[n-1]$, and a Gucycle of $k$-multigraphs of $C_n$.  The latter, in turn, can be identified with the collection of $k$-multisets of $[n]$, via the bijection $j\rightarrow e_j$.  Figure 6 illustrates this for $n=4; k=2$.  In this figure, the ucycle of 0-, 1-, and 2- multisets of $\{1,2,3\}$ is given by $0011010020$, and the Gucycle of 2-multisets of $\{1,2,3,4\}$, encoded by the multigraphs of $C_4$, is pictured.

\begin{figure}[h]
\centering
\includegraphics[width=0.5\textwidth]{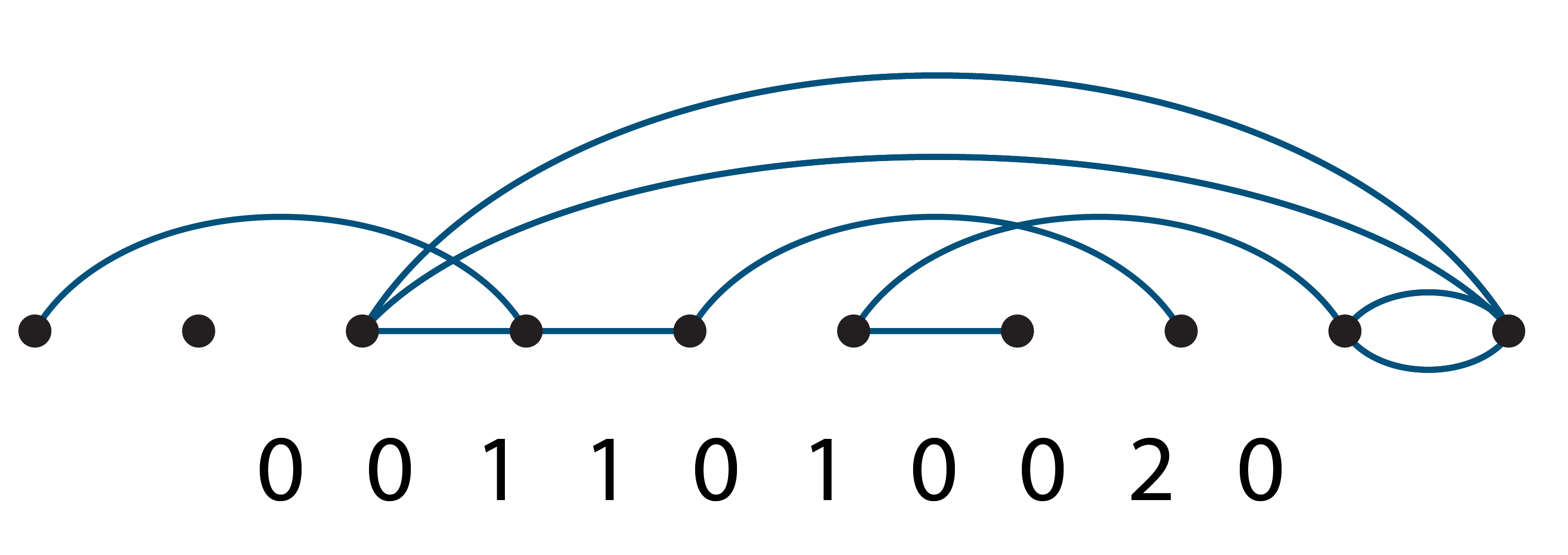}
\caption{Gucycle of the 2-multisets of $[4]$.}
\end{figure}

\end{proof}

\medskip

\n {\it Remark}:  As noted, the traditional way of exhibiting the existence of ucycles is to identify a digraph (the ``arc digraph") and show that it is both balanced ($i(v)=o(v)$ for each $v$) and weakly connected.  We do not do that in the proofs of Theorems 2.5 and 2.6, noting that this work has already been done in \cite{bg} and \cite{cgk}.

\section{Permutations}  Example 5 of Section 1 shows that one needs to enhance the alphabet in order to provide a ucycle of $S_n$ via an order-isomorphic representation.  In \cite{cdg} it was shown that between 1 and $5n$ additional symbols suffice, and it was shown in \cite{jo} that only one extra symbol was actually needed.  In \cite{hh} $s$-ocycles of $S_n$ were shown to exist for suitable values of $s$.  Also in \cite{cdg} a ucycle was exhibited with $n$ symbols provided that the $n-1$ overlaps were order-isomorphic but not identical.  For $n=3$, such a ``chameleon" ucycle, in which vertices ``change their colors," is given by 
\[132\rightarrow312\rightarrow123\rightarrow231\rightarrow321\rightarrow213.\]

We recall that the {\it permutation graph} of $\pi\in S_n$ has vertex set $[n]$ with an edge present between $i$ and $j$ if $i<j$ but $\pi(i)>\pi(j)$, i.e., if $\{i,j\}$ is an inversion.  For example $K_4$ is the permutation graph of $\pi=4321$.  In this section we represent an permutation via its {permutation graph} and exhibit the fact that the class of permutation graphs can be placed in a Gucycle for each $n$.  Figure 7 illustrates this for $n=3$, where the Gucycle lists the permutations in the order 
\[321\rightarrow231\rightarrow312\rightarrow213\rightarrow123\rightarrow132.\]

\begin{figure}[h]
\centering
\includegraphics[width=0.5\textwidth]{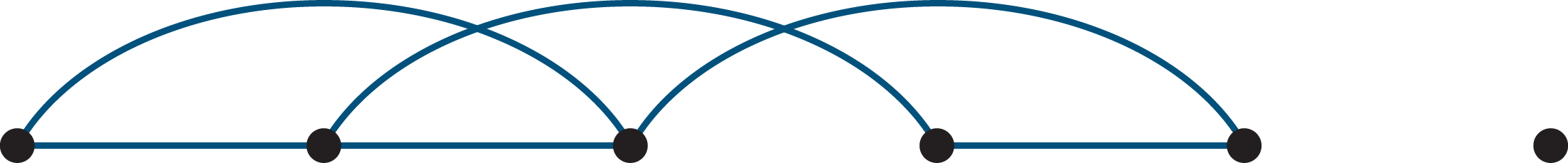}
\caption{Gucycles of all permutations of $[3]$.}
\end{figure}

%
%

\begin{thm} For each $n$, the set $S_n$ of all permutations on $n$ elements can be placed in a Gucycle via their permutation graphs.
\end{thm}

\begin{proof}  We define the an arc digraph as follows:  Let $V$ be the set of all permutation graphs on $\{1,2,\ldots,n-1\}$.  The $n$ edges emanating from vertex $\pi=\pi_1\ldots \pi_{n-1}$ are labeled according as how the index $n$ is inserted into $\pi$, and the vertices that they point towards are labeled as the reduction of the edge label, minus the index 1, to the set $[n-1]$.  For example, there are two edges from $\pi=123$ to $\eta=312$, and these are labeled as 4123 and 1423; the edge between $\pi$ and $\nu=132$ is labeled as 1243; and, finally, the edge between $\pi$ and $\mu=123$ is labeled as 1234.  In general, the insertion of the new element must be done in a way that respects the inversion structure of the starting vertex if one is to exploit the permutation graph representation that we are employing. To give another example, he degree structure of the vertex 132 is shown in Figure 8.

\begin{figure}[h]
\centering
\includegraphics[width=0.7\textwidth]{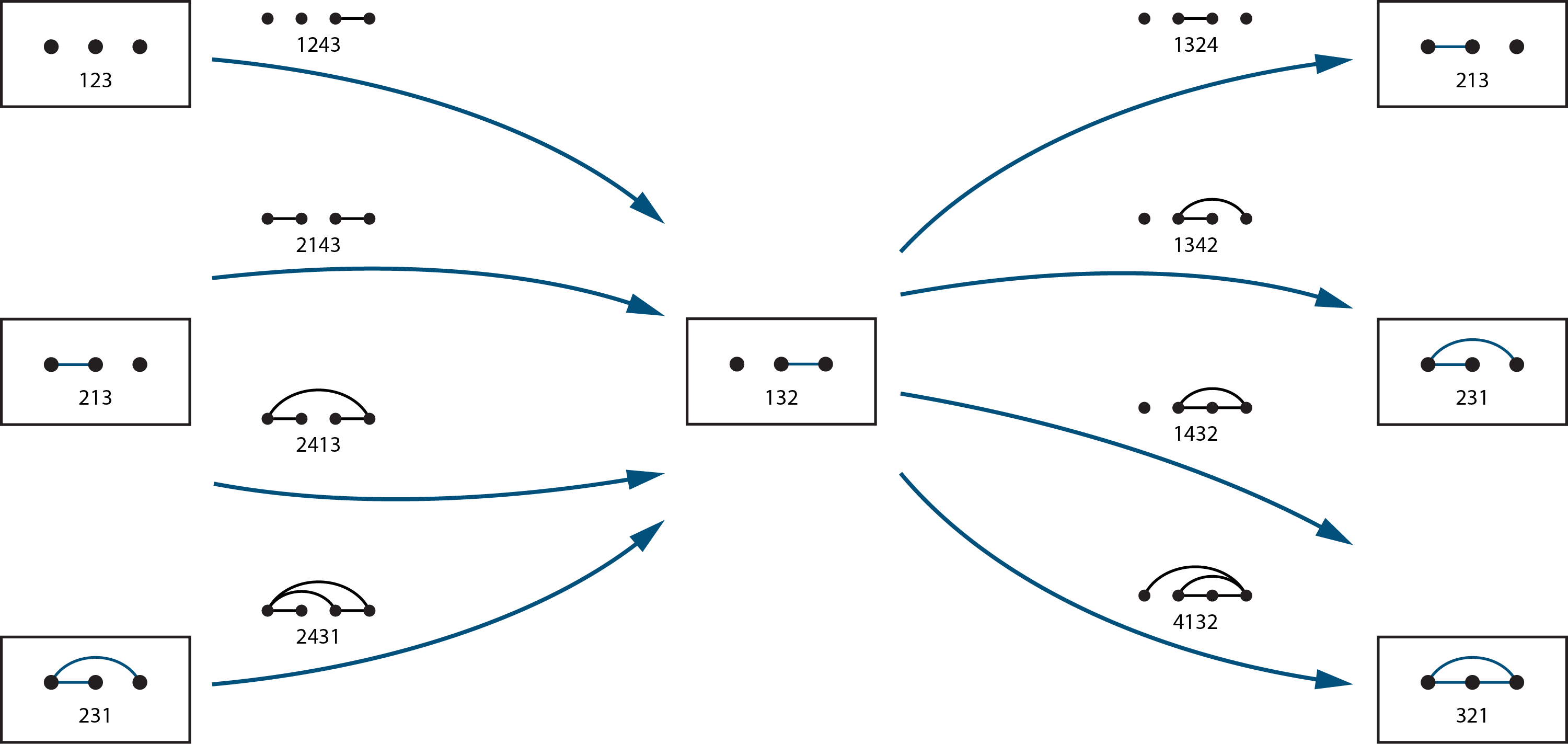}
\caption{The degree structure of the vertex 132 for $n=4$}
\end{figure}

It is clear that the outdegree of $\pi$ is $n$ for each vertex.  We note that the ``overlap" between adjacent vertices is not of the customary form; for example $\pi$ above ends in 23, which is neither equal, nor order isomorphic, to the starting segment 31 of $\eta$.  In general, we have that $\pi_2$ follows $\pi_1$ if the {\it inversion structure} of $\{1,2,\ldots,n-2\}$ in $\pi_2$ is the same as the inversion structure of $\{2,\ldots,n-1\}$ in $\pi_1$.  This is in contrast to (but echoes) the existence of an edge in the deBruijn cycle scheme, in which word $w_2$ follows word $w_1$ if the {\it letters} in the positions $\{1,2,\ldots,n-2\}$ of $w_2$ are the same as the letters $\{2,\ldots,n-1\}$ of $w_1$.

Turning to the indegree of $\pi$, the $n$ edges pointing towards vertex $\pi=\pi_1\ldots \pi_{n-1}$ are labeled according as how the index $1$ is inserted into $\pi'=(\pi_1+1)\ldots (\pi_{n-1}+1)$, with corresponding vertices equaling the edge label minus the index $n$.  The indegree of $\pi$ is thus $n$ as well.  For example, the vertices 123, 213, and 231 point towards 123, with edge labels 1234; 2134; and 2314\&2341 respectively.

We claim that the digraph $D$ defined above is weakly connected and exhibit this easily by creating a path between any starting vertex and the identity permutation (that has an empty permutation graph).  To do this, we merely insert the symbol $n$ at the end of a vertex to create a vertex $\pi$ in which for which $\pi_{n-1}=n-1$.  This process is repeated until all the symbols are in their natural order.  For example, we have
\[43152\rightarrow32415\rightarrow21345\rightarrow12345.\]

Since $D$ is balanced ($i(v)=o(v)\ \forall v$) and weakly connected, it is Eulerian, and the Eulerian circuit spells out the Gucycle.

\end{proof}

\medskip

The reader will recall that, in the context of permutations, an involution is a permutation that is its own inverse, and which consists therefore of transpositions and fixed points, i.e. 1- and 2-cycles.  To represent an involution graphically, perhaps the simplest device is to do so using unions of $K_1$s (the fixed points) and $K_2$s (the transposition).  The {\it transposition graph} of an involution is the labeled graph thus obtained.  The next main result of this section is that all involutions on $n$ elements can be Gucycled; an example for $n=4$, the ten involutions of [4] can be Gucycled as seen in Figure 9, in the order 1324, 2143, 4321, 2134, 4231, 1432, 3412, 3214, 1234, 1243. 

\begin{figure}[h]
\centering
\includegraphics[width=0.7\textwidth]{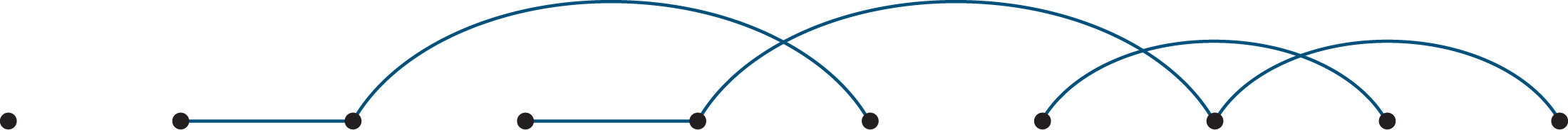}
\caption{Gucycle of the 10 involutions on $\{1,2,3,4\}$}
\end{figure}

%
%
%

\begin{thm} The ``transposition graphs" of all involutions of $S_n$ can be placed in a Gucycle.
\end{thm}
\begin{proof}  Let the vertices of a digraph be represented by the transposition graphs of involutions on $n-1$ elements.  Now the element $n$ can preserve the involution if it pairs up with a previous fixed point, or else forms a fixed point of its own.   Accordingly, the outdegree of any vertex $v$ equals the number of fixed points of $v$ plus one.  Likewise, given an involution $\pi$ on $\{2,3,\ldots,n\}$ the element $1$ might have formed a pair with any of the fixed points of $\pi$ or else been in a fixed point by itself.  This shows its indegree is the same as its outdegree.  Finally, we can easily show weak connectedness by successively adding the fixed point $n$ to any vertex until we end with the identity permutation. Figure 10 illustrates the arc digraph for $n=4$.\end{proof}

\begin{figure}[h]
\centering
\includegraphics[width=0.7\textwidth]{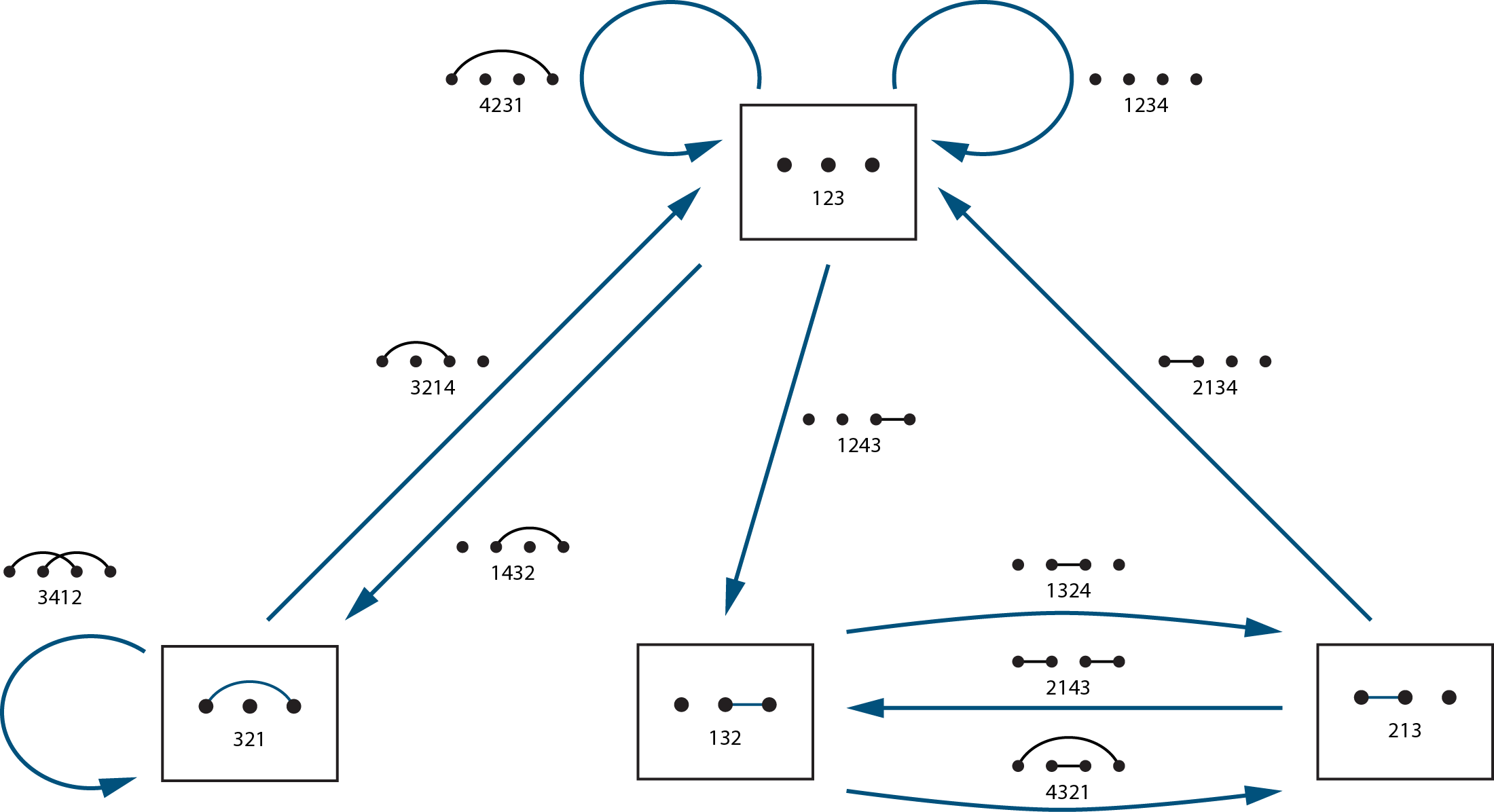}
\caption{Arc digraph for involutions on {1,2,3,4}}
\end{figure}

\section{Partitions}
Let $B(n)$ denote the ordered Bell numbers.  In \cite{cdg} the authors showed that for $n\ge 4$ a ucycle exists for  the $B(n)$ partitions of $[n]$ into an arbitrary number of parts using an enhanced alphabet.  See \cite{ds} and \cite {hksg} for other results for partition ucycles.  Here we show that for Gucycles, the result from \cite {cdg} holds for $n\ge 1$ and that no alphabet augmentation is necessary.  We will represent a set partition by a union of complete subgraphs of $K_n$.  The subgraphs will, of course, be on the elements of the parts.  For example, the partition $134|256|7$ will be represented by two $K_3$s, on the vertices $\{1,3,4\}$ and $\{2,5,6\}$; and the isolated vertex 7.

\begin{thm} For each $n\ge1$, there exists a Gucycle of the partitions of $[n]$ into an arbitrary number of parts.
\end{thm}
\begin{proof} The result is obvious for $n\le 2$.  For $n\ge3$, we define a digraph as having vertex set equal to the graph representations of set partitions of $[n-1]$.  If the partition corresponding to a vertex has $r$ parts, then the element $n$ can either augment the $r$ complete graphs, or else form an isolated vertex.  These are the out-edge labels, which point towards the vertices (unions of complete graphs) induced by the elements $\{2,3,\ldots,n\}$.  This yields $o(v)=r+1$ for any vertex with $r$ parts.  On the other hand, to see that $i(v)=r+1$ as well, we note that the element `1' could have been in a part by itself, or in one of the $r$ parts of $\{2,3,\ldots,n\}$.  Weak connectedness is easy to establish via the path of any vertex to $1|2|\ldots|n-1$, with empty graph.  To accomplish this, we consecutively add the vertex $n-1$ until we reach the desired destination.  This may be seen by the example
\[12345\rightarrow1234|5\rightarrow123|4|5\rightarrow12|3|4|5\rightarrow1|2|3|4|5.\]
The arc digraph and Gucycle for $\cp(3)$, and the Gucycle for $\cp(4)$ may be seen in Figures 11, 12, and 13 respectively.\end{proof}
\vfill\eject
\begin{figure}[h]
\centering
\includegraphics[width=0.7\textwidth]{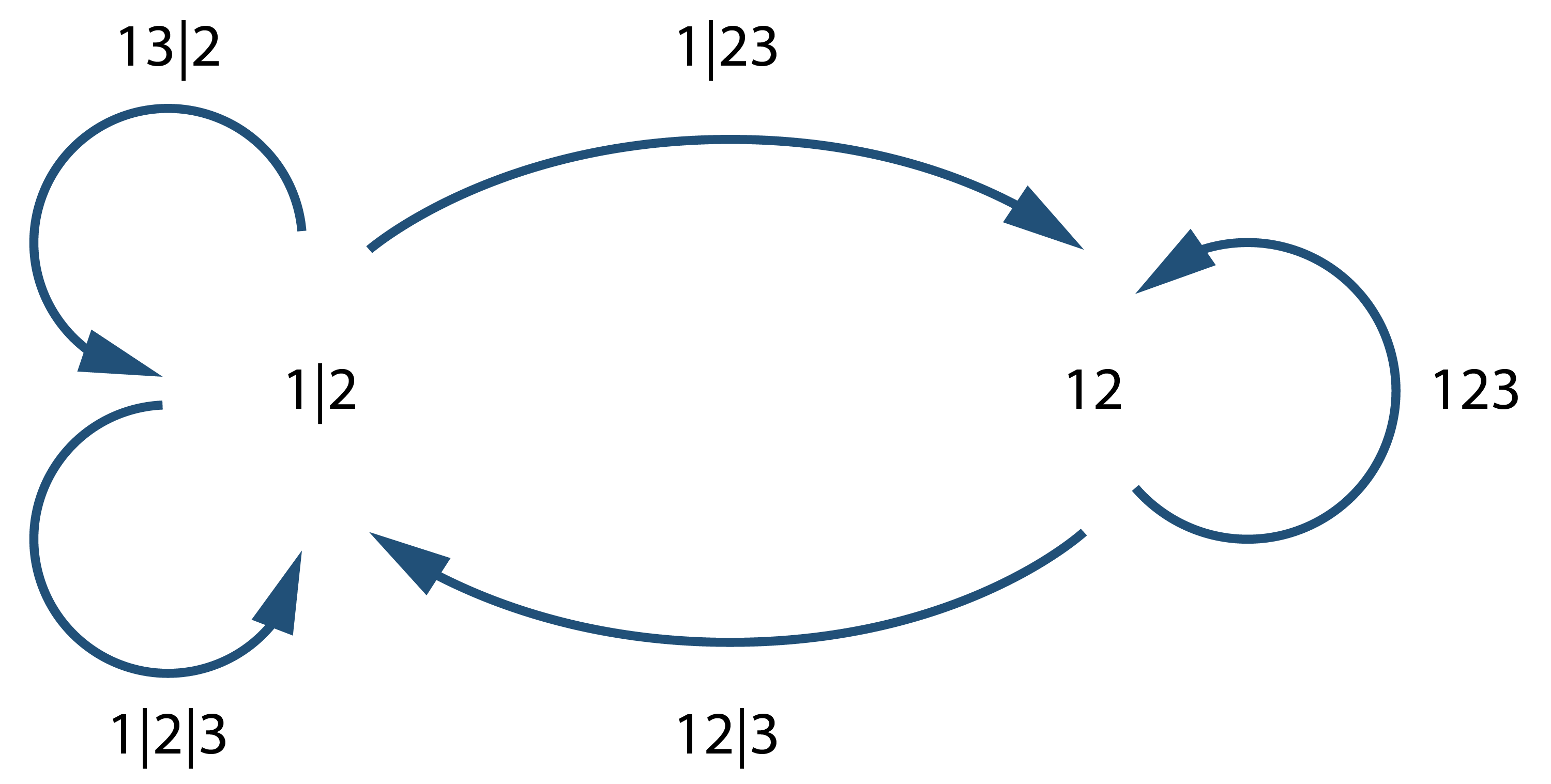}
\caption{Arc digraph for $\cp(3)$}
\end{figure}

\begin{figure}[h]
\centering
\includegraphics[width=0.5\textwidth]{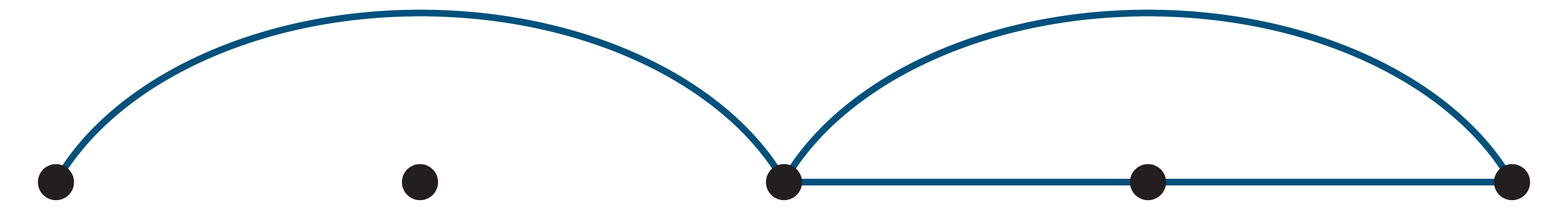}
\caption{Gucycle for $\cp(3)$}
\end{figure}

\begin{figure}[h]
\centering
\includegraphics[width=0.7\textwidth]{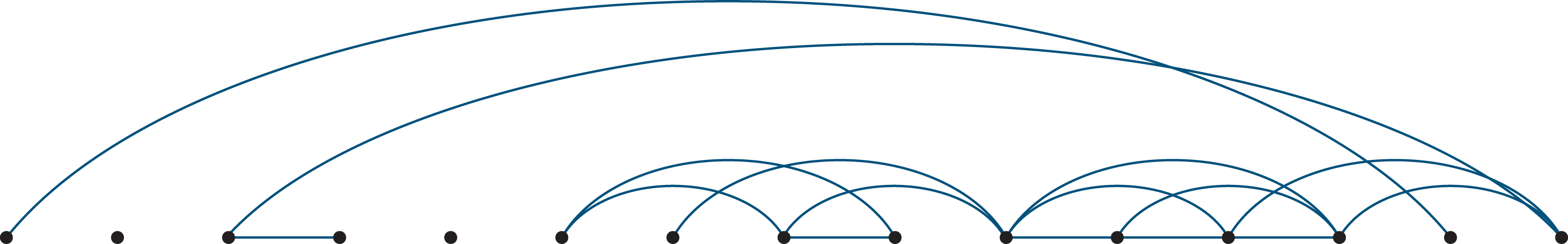}
\caption{Gucycle for $\cp(4)$}
\end{figure}

\section{Open Problems}  The overarching open problem that arises from this paper is the following: ``Find a diversity of examples of other combinatorial structures that admit Gucycles."  In a similar spirit, one may ask ``what combinatorial structures may be fruitfully expresed as labeled graphs?"  Likewise, if we have a subclass of structures that we seek to Gucycle, should we or should we not use the same graph representation used to successfully Gucycle the parent class?  (Notice that we did not use permutation graphs to Gucycle involutions).
\vfill\eject
\section{Acknowledgments} This research was conducted during the Summer of 2019 at the ETSU-UPR (Ponce) REU Program.  Each of the authors was supported by NSF Grant 1852171.  We are most grateful to Christopher Cantwell for his expertise that produced the Figures.

\end{document}